\documentclass[11pt]{article}
\usepackage{amsmath}
\usepackage{amssymb}
\usepackage{bbm}

 \newcommand{\beq}{\begin{equation}}
\newcommand{\eeq}{\end{equation}}

\newtheorem{theorem}{Theorem}[section]

\newenvironment{proof}{\medbreak\noindent{\it Proof.}\rm}{\hfill$\square$\rm}
\newcommand{\cO}{{\mathcal O}}
\newcommand{\cD}{{\mathcal D}}
\newcommand{\cR}{{\mathcal R}}
\newcommand{\PSH}{{\operatorname{PSH}}}
\newcommand{\Cn}{{\mathbb  C\sp n}}
\newcommand{\D}{{\mathbb  D}}
\newcommand{\Zp}{{\mathbb  Z}_+}
\newcommand{\vph}{\varphi}

\begin{document}
\begin{center}
{\Large\bf Zero Lelong number problem}
\end{center}

\begin{center}
{\large Alexander Rashkovskii}
\end{center}

\begin{abstract}
We discuss several related problems on residual Monge-Amp\`ere masses of plurisubharmonic functions. The note is based on the author's talk at the 27th Congress of Nordic Mathematicians, March 19, 2016.

\medskip\noindent
{\sl Mathematic Subject Classification}: 32U05, 32U25, 32U35, 32W20
\end{abstract}

\vskip10pt

\section{Introduction} Here we give a detailed exposition of Questions 7 and 8 from a recent list of open problems in pluripotential theory \cite{DGZ16}. The note is based on the author's talk at the 27th Congress of Nordic Mathematicians, March 19, 2016.

\medskip
Recall that a function $u$ in a domain $\omega\subset\Cn$ is plurisubharmonic if it is upper semicontinuous and such that the composition $u\circ\gamma$ is subharmonic in the unit disk $\D$ for any holomorphic mapping $\gamma:\D\to \omega$. By $\PSH(\omega)$ (resp., $\PSH_0$) we denote the collection of all plurisubharmonic functions in $\omega$ (resp., germs of plurisubharmonic functions at $0\in\Cn$).

A germ $u\in\PSH_0$ is said to be {\it singular} if $u(0)=-\infty$.

A basic characteristic of singularity of $u$ is its {\it Lelong number}
$$ \nu_u=\nu_u(0)=\liminf_{z\to 0}\frac{u(z)}{\log|z|};$$
this is the largest number $\nu\ge 0$ such that
\begin{equation}\label{bd}u(z)\le\nu \log|z|+O(1)\end{equation} near $0$.
Equivalently,
\beq\label{prod}\nu_u=dd^cu\wedge(dd^c\log|z|)^{n-1}(0),\eeq
where $d=\partial+\bar\partial$, $d^c=(\partial-\bar\partial)/2\pi i$.

\medskip
If $f\in\cO_0$ is a holomorphic function near $0$, then $\nu_{\log|f|}={\rm mult}_0 f$, the multiplicity of the $f$ at $0$. However, if $f=(f_1,\ldots,f_m)$ is a holomorphic mapping, then $\nu_{\log|f|}=\min_k{\rm mult}_0 f_k$ is far from the multiplicity of the zero.

Let $\PSH_0^*$ be the germs that are locally bounded outside $0$ (i.e., with {\it isolated singularity} at $0$). The complex Monge-Amp\`ere operator $(dd^c u)^n$ is well defined on such a function $u$ \cite{D93}, and we denote by
$$\tau_u= (dd^cu)^n(0)$$
its {\it residual Monge-Amp\`ere mass} at $0$.

For the mappings $f$ with isolated zero, we have ${\rm mult}_0 f= \tau_{\log|f|}$ \cite{D93}.

\bigskip

\section{ Relations between the characteristics of singularity} For a holomorphic mapping $f$ to $\Cn$, by the local B\'ezout's theorem,
 $${\rm mult}_0 f \ge \prod_k{\rm mult}_0 f_k\ge (\min_k{\rm mult}_0 f_k)^n.$$

\begin{theorem} {\rm \cite{D93}\ } If $u\in \PSH_0^*$, then $\tau_u\ge \nu_u^n$.
\end{theorem}

\begin{proof} This follows from relations (\ref{bd})-(\ref{prod}) and Demailly's Comparison Theorem \cite{D93}
$$ {\rm (CT)}\qquad u_j\le v_j+O(1) \Rightarrow \bigwedge_j dd^c u_j\,(0)\ge \bigwedge_j dd^c v_j\,(0) $$
applied to $u_j=u$, $v_j=\nu_u\log|z|$.
\end{proof}

\medskip
No reverse bound is possible: take $u=\max\{k\log|z_1|,\log|z_2|\}$, then $\nu_u=1$ while $\tau_u=k$.

\bigskip\noindent
 {\bf Problem 1} ({\it Zero Lelong Number Problem}; V. Guedj, A. Rashkovskii, 1999):
 {\sl Is the implication }$$(P1)\qquad\nu_u=0\ \Rightarrow\ \tau_u=0$$
{\sl true whenever $(dd^cu)^n$ is well defined (e.g., for $u\in \PSH_0^*$)?} (This is Question 7 from \cite{DGZ16}.)

\medskip
Note that even in $n=2$ the condition $\nu_u=0$ does not imply $dd^cu\wedge dd^cv\,(0)=0$ for any $v$. For example, if $u=\max\{-|\log|z_1||^{1/2},\log|z_2|\}$ and $v=\log|z_1|$, then $\nu_u=0$ and $dd^cu\wedge dd^cv=\delta_0$.

\bigskip

\section{ Finite {\L}ojasiewicz exponent}
Denote
$$ \gamma_u=\limsup_{z\to 0}\frac{u(z)}{\log|z|},$$
{\it the {\L}ojasiewicz exponent of $u$} at $0$.

By (CT), we get

\begin{theorem} Finite {\L}ojasiewicz exponent
$$ (FLE)\qquad \gamma_u<\infty$$
 implies (P1) is true.
 \end{theorem}

\medskip\noindent
{\bf Examples:} \begin{enumerate}
\item Any {\it analytic singularity} $u=\log|f|+O(1)\in \PSH_0^*$ has FLE (the classical {\L}ojasiewicz exponent of the mapping $f$).

\item {\it Multicircled singularities} $u(z)=u(|z_1|,\ldots,|z_n|)+O(1)\in \PSH_0^*$ have FLE \cite{R00bis}, \cite{R01}.

\end{enumerate}

\bigskip

\section{ Greenifications}
FLE condition might look too restrictive.
But, for $n=1$, any (pluri)subharmonic germ $u$ at $0$ represents as $u(z)=g_u+ v(z)$, where $g_u(z)=\nu_u\,\log|z|$ and  $v$ is a (pluri)subharmonic function with zero Lelong number, so $\nu_{g_u}=\nu_u$ and $g_u$ definitely has FLE.

\medskip

For $n\ge 1$, a 'greenification' \cite{R06} of $u$ is defined in a neighborhood $\omega$ of $0$ as  $$g_u(z):=\limsup_{x\to z}\sup\{v(x):\:  v\in \PSH(\omega),\ v\le 0,\ v\le u+O(1)\ {\rm near\ } 0\}.$$
If $u\in \PSH_0^* $, then $(dd^c g_u)^n=0$ on $\omega\setminus 0$; in other words, $g_u$ is a {\it maximal singularity}. In addition, $\nu_{g_u}=\nu_u$ and $\tau_{g_u}=\tau_u$.

\begin{theorem} {\rm \cite{W05}, \cite{R06}} For $u\in\PSH_0^*$, $g_u\equiv 0 $ if and only if $\tau_{u}=0$.
\end{theorem}

So: (P1) is equivalent to the following question: {\sl For $u\in\PSH_0^*$, does $\nu_u=0$ imply $ g_u\equiv 0$?}

\bigskip\noindent
{\bf Problem 1.1:} {\sl Construct a maximal singularity $\vph\in \PSH_0^*$ with $\gamma_\vph=\infty$.}

\bigskip
{\it Remark:} There exist $u\in \PSH(\omega)$ with non-isolated, maximal singularity at $0$ and well-defined Monge-Amp\`ere operator $(dd^cu)^n$ \cite{R13}; moreover, such a function can have the set $ \{u=-\infty\}$ dense in $\omega$ \cite{ACH}.

\medskip
By removing the condition of isolated singularity, one gets

\bigskip\noindent
{\bf Problem 1.2.} {\sl For $u\in\PSH_0$, is the implication }$$(P1')\qquad\nu_u=0\ \Rightarrow\ g_u\equiv 0$$
{\sl true?}

\medskip

More classes with FLE property to come when considering Problem 2 below.

\bigskip

\section{ Intermediate Lelong numbers}
For $u\in \PSH_0^*$, denote
$$ e_k= e_k(u)=(dd^cu)^k\wedge(dd^c\log|z|)^{n-k}(0),\quad k=0,\ldots,n,$$
so $e_0=1$, $e_1=\nu_u$, $e_n=\tau_u$. If $u=\log|f|$ for a holomorphic mapping $f$ with the zero set of codimension $k$, then $e_k$ equals its multiplicity at $0$.

\medskip

As follows from \cite{Ce04}, these {\it intermediate Lelong numbers} satisfy
$$e_k^2\le e_{k-1}\cdot e_{k+1}$$
which was noticed in \cite{DH14} (in analytic setting $u=\log|f|$, it was established in \cite{Te1}).

\begin{theorem} {\rm \cite{DH14}}
$ e_1=0$ implies $e_k=0$ for all positive $k<n$.
\end{theorem}

\bigskip

\section{ Demailly's approximations}
Problem 1 may be approached by approximating $u$ by functions for which the (affirmative) answer to Problem 1 is known:
e.g., by those with analytic singularities.

We recall a procedure for analytic approximations due to Demailly \cite{D2}. Given $u\in\PSH(\omega)$ and $k\in\Zp$, 
let $\{f_{k,m}\}_m$ be an orthonormal basis of the weighted Hilbert space
$$H_k(u)=\{f\in\cO(\omega):\: \int_\omega |f|^2 e^{-2ku}\,dV<\infty\}.$$
Then the functions
$$\cD_ku=\frac1{2k}\log\sum_m |f_{k,m}|^2\in PSH(\omega)$$
satisfy $$u\le \cD_ku+\frac{C}{k}$$
and converge to $u$ as $k\to\infty$ (in $L_{loc}^1$ and pointwise).
Moreover, $\nu_{\cD_ku}\to \nu_u$.

\medskip
Assume $u\in \PSH_0^*$, then $\cD_ku \in \PSH_0^*$ have analytic singularities. The condition
$ \nu_u=0$ implies $ \nu_{\cD_ku} =0$, which in turn, since $\cD_ku$ have analytic singularities, gives us $\tau_{\cD_ku}=0$, and it remains to relate these to $\tau_u$.

\bigskip\noindent
{\bf Problem 2} {\it (Demailly)}: {\sl Is the convergence}
$$(P2)\qquad \tau_{\cD_ku}\rightarrow \ \tau_u$$
{\sl true?} (Question 8 from \cite{DGZ16}.)

\bigskip

It is known \cite{Bl09} that the functions $$ \cD_{2^k}u + \frac{C}{2^{k+1}}$$
decrease to $u$ and so, by \cite{D93},
$ \left(dd^c\cD_{2^k}u\right)^n$ converge to $ (dd^cu)^n$ as measures (which does not guarantee convergence of their masses at $0$).

\bigskip

\section{ When (P2) is true} For some classes of functions, convergence (P2) is known. Namely, this is so for:

\medskip

1. {\sl Analytic singularities} \cite{BFJ07} $u=c\log|f|+O(1)\in \PSH_0^*$,  where $f$ are holomorphic mappings with isolated zero.

\medskip Such functions are a particular case of

2. {\it Exponentially H\"{o}lder continuous} functions \cite{BFJ07} $\vph\in\PSH_0^*$,
$$ e^{\vph(x)}-e^{\vph(y)}\le A|x-y|^\beta,\quad \beta>0.$$

\medskip A more general class are:

3. {\sl Tame singularities} : $\vph\in\PSH_0^*$ with the property that there exists $C>0$ such that $\forall t>C$ and
every $f\in\cO_0$ the condtition $|f|e^{-t\vph}\in L_{loc}^2$ implies $\log|f|\le (t-C)\vph+O(1)$. These functions are characterized \cite{BFJ07} by the inequalities
$$u+O(1)\le \cD_ku \le (1-C/k)u+O(1),$$
so (P2) follows from (CT).

\medskip
And even more generally, the convergence holds for

4. {\sl Asymptotically analytic singularities} \cite{R13}: $\forall\epsilon>0$ $\exists\vph_\epsilon$ with analytic singularities such that
$$(1+\epsilon)\vph \le \vph_\epsilon \le (1-\epsilon)\vph.$$

\medskip
In particular, any isolated multicircled singularity is asymptotically analytic \cite{R13}.

\begin{theorem} {\rm \cite{R13}} $\vph\in \PSH_0^*$ is asymptotically analytic if and only if the greenifications $g_{\cD_k\vph}$ satisfy
\beq\label{ratio}g_{\cD_k\vph}/g_\vph\rightarrow 1\eeq
uniformly on $\omega\setminus 0$.
\end{theorem}

\medskip
Since the greenifications of analytic singularities are continuous \cite{Za}, convergence (\ref{ratio}) implies $g_\vph\in C(\omega)$ for asymptotically analytic $\vph$.

\bigskip\noindent
{\bf Problem 2.1:} {\sl Construct $u\in \PSH_0^*$ with discontinuous $g_u$.}

\medskip\noindent
{\bf Problem 2.2:} {\sl Construct $u\in \PSH_0^*$ whose singularity is not asymptotically analytic.}

\bigskip

The {\it  type of $u\in\PSH_0$ relative to a maximal weight } $\vph\in \PSH_0^*$ \cite{R13} is
$$\sigma(u,\vph)=\liminf_{z\to 0}\frac{u(z)}{\vph(z)}.$$
Equivalently, it is the largest $\sigma\ge 0$ such that $u(z)\le \sigma\vph(z)+O(1)$.

\medskip\noindent
{\bf Example:} $\nu_u=\sigma(u,\log|z|)$. More generally, let $\phi_a(z)=\max_i a_i^{-1}{\log|z_i|}$, $a_i>0$, then $\sigma(u,\phi_a)$ is Kiselman's directional Lelong number \cite{Kis94} in the direction $a=(a_1,\ldots,a_n)$.

\medskip
It is known that the Lelong numbers of $\cD_ku$, both classical and directional, converge to those of $u$ \cite{D2}, \cite{R01}, and the same do their log canonical thresholds \cite{DK}.

\begin{theorem} {\rm \cite{R13}} The types $\sigma(\cD_ku,\vph)\rightarrow \sigma(u,\vph)$ for any $u\in \PSH_0$ if and only if $\vph$ has asymptotically analytic singularity.
\end{theorem}

\section{ Functions with (P2) property}

\begin{theorem} {\rm \cite{R13}} For $u\in \PSH_0^*$,
TFAE:
\begin{enumerate}
\item[(i)] $\sup_k \tau_{\cD_ku}=\tau_u$;
\item[(ii)] $\lim_{k\to\infty} \tau_{\cD_ku}=\tau_u$;
\item[(iii)] $\inf_k g_{\cD_ku}=g_u$;
\item[(iv)] $\lim_{k\to\infty} g_{\cD_ku}=g_u$;
\item[(v)] there exist analytic singularities $\vph_j\ge u$ such that $\tau_{\vph_j}\to\tau_u$;
\item[(vi)] there exist maximal analytic singularities $\vph_j$ decreasing to $g_u$;
\item[(vii)] there $s_k>0$ and divisorial valuations $\cR_k$, $k=1,2,\ldots$, such that
$$ \sigma(v,g_u)=\inf_k s_k\cR_k(v)\quad\forall v\in PSH_0.$$
\end{enumerate}
\end{theorem}

\noindent
{\bf Problem 2.3:} {\sl Is (P2) true for $u$ with FLE?}

{\footnotesize

Tek/nat, University of Stavanger, 4036 Stavanger, Norway

alexander.rashkovskii@uis.no % Author's institution address, e-mail address.

\end{document}